%% file: sparse_net.tex
\documentclass[12pt]{msml2020} % Anonymized submission
\title[On the stable recovery of deep  structured linear networks under sparsity constraints]{On the stable recovery of deep structured linear networks under sparsity constraints}
\usepackage{times}
\usepackage{amsmath}
\usepackage{amsfonts}
\usepackage{amsbsy}
\usepackage{amssymb}
\usepackage{relsize}
\usepackage{graphicx}
\usepackage{calc}
\usepackage{amstext}
\usepackage{bbm}
\usepackage{mathrsfs}
\usepackage{authblk}
\usepackage{color}
\newcommand{\red }{\color{red}}
\newcommand {\green}[1] {\textcolor[rgb]{0.0,0.6,0}{#1}}

\usepackage{pict2e}
\linethickness{0.2mm}

\usepackage{graphicx}

\input{math_macro.tex}

 \msmlauthor{
 \Name{Fran\c{c}ois Malgouyres} \Email{malgouyres@math.univ-toulouse.fr}\\
 \addr Institut de Math\'ematiques de Toulouse ; UMR5219 \\ 
	Universit\'e de Toulouse ; CNRS   \\
	UPS IMT F-31062 Toulouse Cedex 9, France\\
	and \\
	Institut de Recherche Technologique Saint-Exup\'ery\\
 }

\begin{document}

\maketitle

\begin{abstract}

  We consider a deep structured linear network under sparsity constraints. We study sharp conditions guaranteeing the stability of the optimal parameters defining the network. More precisely, we provide sharp  conditions on the network architecture and the sample under which the error on the parameters defining the network  scales linearly with the reconstruction error (i.e. the risk). Therefore, under these conditions, the weights obtained with a successful algorithms are well defined and only depend on the architecture of the network and the sample. The features in the latent spaces are stably defined. The stability property is required in order to interpret the features defined in the latent spaces. It can also lead to a guarantee on the statistical risk. This is what motivates this study.  
  
The analysis is based on the recently proposed Tensorial Lifting. The particularity of this paper is to consider a sparsity prior. This leads to a better stability constant. As an illustration, we detail the analysis and provide sharp stability guarantees for convolutional linear network under sparsity prior. In this analysis, we distinguish the role of the network architecture and the sample input. This  highlights the requirements on the data in connection to parameter stability.
\end{abstract}


\input{content.tex}

\input{example_spars1.tex}

% Acknowledgments---Will not appear in anonymized version
\acks{Francois Malgouyres is funded by the project DEEL (\url{https://www.deel.ai/}) and by the ANITI (\url{https://aniti.univ-toulouse.fr/index.php/en/}). The author would like to thank Joseph Landsberg for all his remarks.}

\bibliography{ref}

\input{appendices.tex}

\end{document}

%% file: math_macro.tex
\newcommand {\one} {\mathbbm{1}}
\newcommand {\RR} {\mathbb R}

\newcommand {\NN} {\mathbb N}
\newcommand {\N}[1] { [\![#1 ]\!] }

\newcommand {\calA} {\mathcal A}

\newcommand {\calS} {\mathcal S}
\newcommand {\calT} {\mathcal T}

\newcommand {\calD} {\mathcal{D}}
\newcommand {\cald} {\Delta}
\newcommand {\SF} {{\mathcal M}}
\newcommand {\SSp} {{\mathcal P}(\iS)}
\newcommand {\SSpo} {{\mathcal P}(\iSo)}
\newcommand {\Mod}[1] {{\SF}^{#1}}
\newcommand {\TS} {\RR^{S^H}}   % tensor space (TSpace)
\newcommand {\wS} {\RR^{S\times H}}    % collection of kernels space (hSpace)
\newcommand {\wSdiag} {\RR^{S\times H}_{\mbox{\tiny diag}}}    % collection of kernels space (hSpace) with equal norms
\newcommand {\iS} {\N{S}^H}  % set of the indices in a tensors (iSet)
\newcommand {\iSo} {\{1,\cdots,S\}^H}  % set of the indices in a tensors (iSet)
\newcommand {\ibf} { \mathbf{i}}      % index in a tensor
\newcommand {\Ibf} { \mathbf{I}}      % index in a tensor
\newcommand {\jbf} { \mathbf{j}}      % index in a tensor
\newcommand {\ebf} { \mathbf{e}}      % index in a tensor
\newcommand {\wbf} { \mathbf{w}}      % other collection of kernels
\newcommand {\vbf} { \mathbf{v}}      % other collection of kernels
\newcommand {\pbf} { \mathbf{p}}      % other collection of kernels
\DeclareMathOperator {\rk} {rk}
\newcommand {\RH}[1] {\rk\left( #1 \right)}   % rank of a tensor
\DeclareMathOperator {\myspan} {Span}
\newcommand {\SPAN}[1] { \myspan\left( #1 \right) }
\DeclareMathOperator {\myker} {Ker}
\newcommand {\KER}[1] { \myker\left( #1 \right) }

\newcommand {\DIM}[1] {\dim\left( #1 \right)}   % rank of a tensor
\newcommand {\PROJ}[1] {{ \mathbf{P}}_{ #1 } }   % projection

\newcommand {\tree}{\mathcal G}
\newcommand {\nodes} {\mathcal N}
\newcommand {\edges} {\mathcal E}
\newcommand {\leaves} {\mathcal F}

\newcommand {\PP} {{\mathcal{P}}}
\newcommand {\multiconv}[2] { {\mathcal T}^{#1}(#2)\, }

\newcommand {\codeset} {\mathbb R^{N  |\leaves|}}

\DeclareMathOperator {\supp} {supp}
\newcommand {\SUPP}[1] {\supp\left( #1 \right)}
\newcommand {\class}[1] {\{ #1 \}}
\newcommand {\SNSP} {sparse-deep-NSP }
\newcommand {\SNSPlong} {sparse-deep-Null Space Property }
\newcommand {\Id}{Id}
\newcommand {\Idl}{M_0}

\newtheorem{defi}{Definition}
\newtheorem{prop}{Proposition}
\newtheorem{thm}{Theorem}

\newtheorem{informTh}{Informal theorem}

%% file: content.tex
%%%%%%%%%%%%%%%%%%%%%%%%%%%%%%%%%%%%%%%%%%%%%%%%%%%%%%%%%%%%%%%%%%%%%%%%%%%%%%%
\section{Introduction}
\subsection{The stability property}
Artificial neural networks have improved the state of the art and continue to improve it in a large number of applications in science and technology. Their empirical success far exceeds the  understanding of their theoretical properties. In particular, despite the very significant efforts of a very active research community, some behaviors remain partially understood: Why do optimization algorithms find good solutions? Why do over-parameterized neural networks retain good generalization properties? What classes of functions can be approximated by neural networks? With which minimal network architecture?

The work presented in this paper is of a theoretical nature and focuses on a stability property for the parameters leading to a low objective function. The statements are for a regression problem. To explain this stability property in a simplified context, we consider a  parameterized family of  functions $f_\wbf$ (e.g. neural networks), the parameter being $\wbf$ ; the parameter space is equipped with a metric\footnote{To be accurate, the metric is defined between equivalence classes reflecting invariance properties of the family $f$. For instance, in the case of neural networks, we would like to consider weight rescaling and/or neurons re-arrangement.} $d$; we consider a sample $(x_i,y_i)_{i=1..n}$ of size $n\in\NN$. The stability statement then takes the following form.

\begin{informTh}{\bf Stability Guarantee}\label{inform}

If a certain {\em condition} on the family $f$ and the sample is satisfied then we have the following {\em stability property}:

There exists $C>0$ such that for $\eta$ sufficiently small:
For any $\wbf$ and $\wbf'$ such that
\[ \sum_{i=1}^n \| f_\wbf(x_i) - y_i  \|^2   \leq \eta \qquad\mbox{and}\qquad \sum_{i=1}^n \| f_{\wbf'}(x_i) - y_i  \|^2   \leq \eta
 \]
we have
\[d(\wbf,\wbf') \leq C \eta.
\]
\end{informTh}
Notice first that the above informal theorem provides a sufficient condition guaranteeing the stability property. Subsequently, depending on the nature of the network under consideration, necessary and sufficient conditions or necessary conditions will be stated. The interest of the stability property is that it guarantees:
\begin{itemize}
\item {\bf Feature stability and interpretability: } When the optimal $\wbf$ is stable, the features in the latent spaces and the output of the network are stably defined in the sense that the parameters $\wbf$ and $\wbf'$ for which $\eta$ is small define similar features and output. The features and the output only depend on the value of 
\[\sum_{i=1}^n \| f_\wbf(x_i) - y_i  \|^2 \]
and do not depend on the algorithm used to find $\wbf$. In particular, they do not depend on its initialization, the numerical parameters, the order of the samples in the stochastic algorithm, the numerical tricks etc The parameter $\wbf$ and therefore the function $f_\wbf$ only depends on $f$ (i.e.: the network architecture, for neural networks) and the sample  $(x_i,y_i)_{i=1..n}$. For neural networks, this is a strong guaranty when interpreting the influence of the features on the output.
 \item {\bf Stable recovery: } If we make the additional assumption that the data are generated from the family $f$ for an ideal parameter $\wbf$ (up to an accuracy smaller than $\eta$), then the stability guaranty ensures that any parameter $\wbf'$ for which  
 \[  \sum_{i=1}^n \| f_{\wbf'}(x_i) - y_i  \|^2
 \]
 is sufficiently small is close to the ideal $\wbf$.
 
 The above additional assumption can be provided by approximation theory statement. This is, for instance, the usual argument in compressed-sensing \cite{elad_book}. When solving a linear inverse problem under sparsity constraints, the sparsity hypothesis is not so restrictive because many signals/images classes are compressible. We can expect the same phenomenon to happen for neural networks for which such statements are often referred to as ``expressivity'' or ``expressive power''.  For instance, we can expect to have such guaranties when the neural network approximates a smooth function \cite{boolcskei2019optimal,guhring2019error,gribonval:hal-02117139}. 
\end{itemize}

\subsection{Existing results on the stable recovery}
Establishing stable recovery guarantees for neural networks is a difficult subject which has not been addressed very often. The subject remains largely unexplored.
To the best of our knowledge conditions guaranteeing the stability property for neural networks have been established in   \cite{arora2014provable,brutzkus2017globally,li2017convergence,sedghi2014provable,zhong2017recovery,MalgouyresLandsbergITW,MalgouyresLandsberg_long}.
A negative statement, exhibiting an unstable configuration when the weights go to infinity is given in \cite{Petersen2020}.

Among them, \cite{brutzkus2017globally,li2017convergence,zhong2017recovery} consider a family of networks with one hidden layer. The article \cite{sedghi2014provable} focuses on the recovery of the parameters defining one layer in a arbitrarily deep networks. The articles \cite{arora2014provable,MalgouyresLandsbergITW,MalgouyresLandsberg_long} consider networks without depth limitation.

In \cite{brutzkus2017globally}, the authors consider the minimization of the population risk. The input is assumed Gaussian and the output is generated by a network involving one linear layer followed by ReLU and a mean. The number of intermediate nodes is smaller than the input size. They provide conditions guaranteeing that, with high probability, a randomly initialized gradient descent algorithm converges to the true parameters. The authors of \cite{li2017convergence} consider a framework similar to  \cite{brutzkus2017globally}. They show that the stochastic gradient descent converges to the true solution. In \cite{zhong2017recovery}, the authors consider a non-linear layer followed by a linear layer. The size of the intermediate layer is smaller than the size of the input and the size of the output is $1$. They prove that the gradient algorithm minimizing the empirical risk converges to the true parameters, for the particular initialization described in the article.

The authors of \cite{sedghi2014provable} consider a feed-forward neural network and show that, if the input is Gaussian or its distribution is known, a method based on moments and sparse dictionary learning can retrieve the parameters defining the first linear transform. Nothing is said about the stability or the estimation of the other transformations.

The authors of \cite{arora2014provable} consider deep feed-forward networks which are very sparse and randomly generated. They show that they can be learned with high probability one layer after another. However, very sparse and randomly generated networks are not used in practice and one might want to study more versatile structures. 

The article \cite{MalgouyresLandsbergITW} studies deep structured linear networks and uses the same tensorial lifting we use here. This result has been extended in \cite{MalgouyresLandsberg_long}, where necessary and sufficient conditions of stable recovery have been established for a general constraint on the parameters defining the network. In the present article, we specialize the analysis to the sparsity constraint. We also obtain necessary and sufficient conditions of stable recovery. However, we obtain a better stability constant (the constant $C$ in Informal Theorem \ref{inform}). The difference is of the same nature as when the smallest singular value is replaced by a lower RIP constant in compressed sensing \cite{elad_book}.  Moreover, in the analysis dedicated to convolutional linear networks, we separate the hypotheses on the data $(x_i)_{i=1..n}$ and the network architecture. This  highlights the importance of having a full row rank $X$, where $X$ is the concatenation of the data $(x_i)_{i=1..n}$,  and shows the role of the smallest singular value of $X$ in this context. These are the two main contributions of the paper.

Finally, denoting $H$ the number of factors/layers, the approach developed in this paper extends to $H\geq 3$ existing compressed sensing results for $H\leq 2$. In particular, when $H=1$, the considered problems boils down to a compressed sensing problem \cite{elad_book}. When $H=2$ and when extended to other constraints on the parameters $\wbf$, the statements apply to already studied problems such as: low rank approximation \cite{candes2013phaselift}, Non-negative matrix factorization \cite{lee1999learning,donoho2003does,laurberg2008theorems,arora2012computing}, dictionary learning \cite{Jenatton}, phase retrieval \cite{candes2013phaselift}, blind deconvolution   \cite{ahmed2014blind,choudhary2014identifiability,DBLP:journals/corr/LiLSW16}. Most of these papers use the same lifting property we are using. They further propose to convexify the  problem. A more general bilinear framework is considered in \cite{choudhary2014identifiability}.

%The present work describes an alternative analysis, specialized to sparsity constraints, of the results exposed in \cite{MalgouyresLandsberg_long}. Doing so, we obtain better bounds (defined with an analogue of the lower-RIP) and weaker constraints on the model. Its application to sparse convolutional linear networks leads to simple necessary and sufficient conditions of stable recovery, for a large class of solvers. The stability inequality (see Theorem \ref{stabl-rec-network-thm}) only involves explicit and simple ingredients of the problem. The condition on the network architecture is rather strong but takes a simple format. Implementing a test checking if the condition is met is easy and the test only requires to apply the networks as many times as the network has leaves, for every couple of supports.

\subsection{The considered sparse networks}
As in \cite{MalgouyresLandsberg_long}, we consider {\em structured linear networks}. The layers can be {\em convolutional} or {\em feedforward}. The network has at least one hidden layer and can be {\em deep}. The  network is not biased. We give in this section all the notations on networks.

Throughout the paper, we consider $H\geq 2$, $S\geq 2$, $m_0 \hdots   m_{H} \in \NN$ and  write $m_H=m$. We consider a network and assume its architecture fixed. It has $H-1$ hidden layers. The  layer $0$ corresponds to the inputs, the layer $H$ to the output. The hidden layers correspond to the indexes $1,\cdots, H-1$. For $h\in\{0,\cdots, H\}$, $m_h$ is the size of the layer $h$. We assume that the whole network is parameterized by an element of $\wS$, say $\wbf\in \wS$. The architecture of the network is defined by linear mappings
\begin{eqnarray}\label{Mh}
M_h : \RR^S & \longrightarrow & \RR^{m_h\times m_{h-1}} \\
w & \longmapsto & M_h(w) \nonumber
\end{eqnarray}
for $h\in\{1,\cdots, H\}$.  For all $h\in\{1,\cdots, H\}$, the linear part of the transformation that maps the content of the layer $h-1$ to the layer $h$ is parameterized by $\wbf_h\in\RR^S$ and is defined by $M_h(\wbf_h)$.  Modeling the architecture of the network with the operators $M_h$, we can consider many kind of networks. Indeed, depending on the operators $M_h$, the network can include feedforward layers,  convolutional layers and other structured layers tailored to particular structures in the data. The layers might not be fully connected.

The mapping from $\RR^{m_0}$ to $\RR^{m}$ defined by the network is called the  {\em prediction} and it is defined for any $x\in\RR^{m_0}$ by
\[f_\wbf(x) =M_H(\wbf_H)M_{H-1}(\wbf_{H-1}) \cdots M_2(\wbf_2)  M_1(\wbf_1) x.
\]
We use the same notation $f_\wbf$ when applying $f_\wbf$ to every column of $X\in\RR^{m_0\times n}$ and concatenating the results in a matrix in $\RR^{m\times n}$. The abuse of notation is not ambiguous, once in context.

Again, the considered networks do not involve activation functions and biases.  However, as indicated in \cite{MalgouyresLandsberg_long}, the action of the ReLU activation function multiplies the content of any neuron by an element of $\{0, 1 \}$. The choice of the element depends on $\wbf$ and $x$. However, considering $x$ fixed, since $\{0,1\}$ is finite, there is a finite set of possibilities for the action of the ReLU activation function. Said differently, there is a finite number of possibilities for the choice of the neurons that are kept. Therefore, there exists a partition of $\wS$ such that, on every piece of the partition, the action of ReLU is constant. Therefore, on every piece of the partition, the network is a structured linear network as studied in the present paper. Notice moreover that the analysis in \cite{choromanska2015loss,choromanska2015open} take the expectation of the action of ReLU networks (under an un-realistic independence hypothesis) and obtain a {\em structured linear network}. Beside, structured linear network are significantly more general than the {\em deep linear networks} that are often considered (see, among other, \cite{baldi1989neural,kawaguchi2016deep}). 

Throughout the paper, we consider a family of possible supports $\Mod{}\subset\SSpo$, where  $\SSpo$ denotes the set of all  possible supports (the parts of $\iSo$). A classical example is $\Mod{} = \{\calS | \forall h=1..h, |\calS_h|\leq S'\}$, for a given $S'\le S$. We constrain the parameter $\wbf$ to satisfy a sparsity constraint of the form : there exists ${\calS} =(\calS_h)_{ h = 1..H} \in \Mod{}$ such that 
\[\SUPP{{\wbf}} \subset {\calS}
\]
 (i.e.: $\forall h$, $\SUPP{{\wbf}_h} \subset {\calS}_h$). Specializing the analysis to sparsity constraints is one of the main differences between this paper and \cite{MalgouyresLandsbergITW,MalgouyresLandsberg_long}. Sparse networks have been considered in many contexts \cite{ranzato2007efficient,ranzato2008sparse,lee2008sparse,srinivas2017training,louizos2017learning,zhang2016cambricon}; sparse convolutional neural networks have also been considered \cite{liu2015sparse}.

\subsection{The solutions of the problem}

We assume that data are collected in the columns of matrices $X\in\RR^{m_0\times n}$ and $Y\in\RR^{m_H\times n}$.

To establish the stability property, we consider throughout the paper $\overline\calS$ and ${\overline\calS}' \in\Mod{}$, $\overline\wbf$ and $\overline\wbf' \in \wS$ such that
\[\SUPP{{\overline\wbf}} \subset {\overline\calS} \qquad \mbox{and}\qquad \SUPP{{\overline\wbf'}} \subset {\overline\calS}'
\]
and for which
\begin{equation}\label{model_estim-sparse}
\|f_{\overline\wbf}(X) - Y\| = \delta \qquad \mbox{and}\qquad \|f_{\overline\wbf'}(X) - Y\| = \eta
\end{equation}
are small. Generic parameters are denoted without the over-line: $\calS$, $\calS'$, $\wbf$, $\wbf'$ etc

We want to establish a condition guaranteeing that, up to a multiplicative constant, the distance between such $\overline\wbf$ and $\overline\wbf'$ is upper-bound by $\delta+\eta$. As already said, using a true distance would be too restrictive and the true statements involve a distance between equivalence classes of parameters.

\section{Notations and preliminaries on Tensorial Lifting}\label{notation-sec}

Set $\N{H} = \{1,\ldots,H\}$ and $\wS_* = \{\wbf \in \wS| \forall h=1..H, \|\wbf_h\| \neq 0  \}$, where we remind that $\wbf_h\in\RR^S$ contains the parameters defining the transform between layers $h-1$ and $h$. Define an  equivalence relation in $\wS_*$: for any  $\wbf$,  $\vbf \in \wS$, $\wbf \sim \vbf$  if and only if there exists $(\lambda_h)_{h=1..H} \in\RR^H$ such that
\[\prod_{h=1}^H \lambda_h = 1 \qquad \mbox{ and } \qquad \forall h=1..H, \wbf_h = \lambda_h\vbf_h.
\]
Denote the equivalence class of  $\wbf\in\wS_*$ by $[\wbf]$. For any $p\in[1,\infty]$, we denote the usual $\ell^p$ norm by $\|.\|_p$ and  define the mapping   $d_p:\left((\wS_* /\sim) \times (\wS_* /\sim)\right) \rightarrow \RR$ by
\begin{equation}\label{dp_def}
d_p([\wbf], [\vbf]) =  \inf_{\substack{\wbf'\in[\wbf]\cap\wSdiag \\
\vbf'\in[\vbf] \cap \wSdiag }} \|\wbf'-\vbf'\|_p\qquad,\forall \wbf,\, \vbf \in \wS_*,
\end{equation}
where
\[\wSdiag = \{\wbf \in \wS_*| \forall h=1..H, \|\wbf_h\|_\infty = \|\wbf_1\|_\infty  \}.\]
It is proved in \cite{MalgouyresLandsberg_long} that $d_p$ is a metric on $\wS_* /\sim$.

The real valued tensors of order $H$ whose axes are of size $S$ are denoted by $T\in\RR^{S\times \ldots \times S}$. The space of tensors is
 abbreviated  $\TS$. We say that a tensor $T\in\TS$ is  of {\it rank $1$} if and only if there exists a collection of vectors $\wbf\in\wS$ such that,  for any $\ibf = (i_1,\ldots,i_H)\in \iS$,
\[T_\ibf = \wbf_{1,i_1} \ldots \wbf_{H,i_H}.
\]
The set of all the tensors of rank less than $1$ is denoted by $\Sigma_1$. We denote $\Sigma_2=\Sigma_1+\Sigma_1$. Moreover, we parameterize  $\Sigma_1\subset \TS$ using the Segre embedding
\begin{equation}\label{defP}
\begin{array}{rcl}
P: \wS & \longrightarrow & \Sigma_1 \subset \TS \\
\wbf & \longmapsto & (\wbf_{1,i_1}\wbf_{2,i_2} \ldots \wbf_{H,i_H})_{\ibf\in\iS}
\end{array}
\end{equation}

As stated in the next two theorems, we can control the distortion of the distance induced by $P$ and its `inverse'.

\begin{thm}\label{rk1-Ident-thm}{\bf Stability of $[\wbf]$ from $P(\wbf)$, see \cite{MalgouyresLandsberg_long}}

Let $\wbf$ and $\wbf'\in\wS_*$ be such that
%$\|P(\gbf)\|_\infty \leq \|P(\wbf)\|_\infty$ and
$\|P(\wbf')-P(\wbf)\|_\infty\leq \frac{1}{2}  ~ \max\left(\|P(\wbf)\|_\infty,\|P(\wbf')\|_\infty\right) $. For all $p, q\in[1,\infty]$,
\begin{equation}\label{Pinv-lipschitz}
d_p([\wbf], [\wbf']) \leq 7 (HS)^{\frac{1}{p}} \min\left(\|P(\wbf)\|_{\infty}^{\frac{1}{H}-1},\|P(\wbf')\|_{\infty}^{\frac{1}{H}-1} \right) \|P(\wbf)-P(\wbf')\|_{q}.
\end{equation}
\end{thm}

\begin{thm}\label{PLip-thm}{\bf `Lipschitz' continuity of $P$, see \cite{MalgouyresLandsberg_long}}

We have for any $q\in[1,\infty]$ and any $\wbf$ and $\wbf'\in\wS_*$,
\begin{equation}\label{P-lipschitz}
 \|P(\wbf)-P(\wbf')\|_{q} \leq S^{\frac{H-1}{q}} H^{1-\frac{1}{q}}  \max\left(\|P(\wbf)\|_{\infty}^{1-\frac{1}{H}} , \|P(\wbf')\|_{\infty}^{1-\frac{1}{H}} \right) d_q([\wbf], [\wbf']).
\end{equation}
%For any $\wbf$ and $\gbf\in\wS_*$, we have
%\[ \|P(\wbf)-P(\gbf)\|_{+\infty} \leq H \max\left(\|P(\wbf)\|_{\infty}^{1-\frac{1}{H}} , \|P(\gbf)\|_{\infty}^{1-\frac{1}{H}} \right)~ d_{+\infty}([\wbf], [\gbf]).
%\]
\end{thm}

The Tensorial Lifting (see \cite{MalgouyresLandsberg_long}) states that for any $M_1$, \ldots, $M_H$ and any $X$ there exists a unique linear map
\[\calA:\TS \longrightarrow \RR^{m\times n},
\]
such that for all $\wbf\in\wS$
\begin{equation}\label{lifting}
M_H(\wbf_H) \cdots M_1(\wbf_1)X = \calA P(\wbf).
\end{equation}
The intuition leading to this equality is that every entry in $M_H(\wbf_H) \cdots M_1(\wbf_1)X$ is a multivariate polynomial whose variables are in $\wbf$. Moreover, every monomial of the polynomials is of the form $a_\ibf P(\wbf)_\ibf$ for $\ibf\in\iS$, where $a_\ibf$ is a coefficient which depends on $M_1$, \ldots, $M_H$ and $X$. The Tensorial Lifting expresses any deep structured linear network using the Segre Embedding and a linear operator $\calA$. The Segre embedding is non-linear and might seem difficult to deal with at the first sight, but it is always the same whatever the network architecture, the sparsity pattern, the action of the ReLU activation function\ldots These constituents of the problem only influence the lifting linear operator $\calA$.

In the next section, we study what properties of $\calA$ are required to obtain the stable recovery. In Section \ref{conv-tree}, we study these properties when $\calA$ corresponds to a sparse convolutional linear network.
%%%%%%%%%%%%%%%%%%%%%%%%%%%%%%%%%%%%%%%%%%%%%%%%%%%%%%%%%%%%%%%%%%%%%%%%%%%%%%

\section{General conditions for the stable recovery under sparsity constraint}\label{stable-sparse-sec}

From now on, the analysis differs from the one presented in \cite{MalgouyresLandsberg_long}. It is dedicated to models that enforce sparsity. In this particular situation, we can indeed have a different view of the geometry of the problem. In order to describe it, we first establish some notation.

We define a support by $\calS = (\calS_h)_{ h = 1..H}$, with $\calS_h\subset\N{S}$,  and remind that we denote the set of all supports by $\SSp$ (the parts of $\iS$). For a given support $\calS\in\SSp$, we denote
\[\wS_{\calS} = \{ \wbf \in\wS ~|~ \wbf_{h,i} =0 \mbox{, for all }h=1..H \mbox{ and } i\not\in\calS_h   \}
\]
(i.e., for all $h$, $\SUPP{\wbf_h}\subset \calS_h$) and
\[\TS_{\calS} = \{ T \in\TS ~|~ T_\ibf = 0 \mbox{, if }\exists h=1..H \mbox{, such that } \ibf_h\not\in \calS_h  \}.
\]

We also denote by $\PROJ{\calS}$ the orthogonal projection from $\TS$ onto $\TS_{\calS}$. It has the closed-form expression: for all $T\in\TS$ and all $\ibf\in\iS$
\begin{equation}\label{proj_def}
(\PROJ{\calS}T)_\ibf = \left\{\begin{array}{ll}
{T}_\ibf & \mbox{, if } \ibf \in\calS, \\
0 & \mbox{, otherwise.}
\end{array}\right.
\end{equation}

%As explained in the introduction, we assume that there exists a known family of admissible supports $\SF\subset \SSp$, an unknown support $\overline{\calS}\in\SF$ and unknown parameters $\overline{\wbf} \in \wS_{\overline{\calS}}$ that we would like to estimate from the noisy matrix product
%\begin{equation}\label{data_sparse}
%X = M_1(\overline{\wbf}_1)\cdots M_H(\overline{\wbf}_H) + e.
%\end{equation}
%We assume that there exists $\delta\geq 0$ such that the error satisfies
%\begin{equation}\label{delta-sparse}
%\|e\|\leq \delta.
%\end{equation}
%Also, we consider an inexact minimization and assume that we have a way to find ${\overline\calS}'\in\SF$ and ${\overline\wbf}' \in \wS_{{\overline\calS}'}$ 
%\[\eta = \|M_1({\overline\wbf}'_1)\cdots M_H({\overline\wbf}'_H) - X\| \qquad \mbox{ is small.}
%\]
%We remind that, in machine learning problems, $\eta$ represents the risk.

We consider different operators and define for any  $\calS \in\SSp$
\begin{equation}\label{defAS}
\calA_\calS = \calA \PROJ{\calS}.
\end{equation}
We will use later on that for any $\calS$ and $\calS'\in\SF$ and for any $\wbf\in\wS_{\calS}$, or any $\wbf\in\wS_{\calS'}$, or any $\wbf\in\wS_{\calS\cup\calS'}$,  we have
\begin{eqnarray}
\calA_{\calS\cup\calS'} P(\wbf) & = & \calA P(\wbf), \label{ASSpP}\\
 & = &  M_H(\wbf_H)\cdots M_1(\wbf_1)X. \nonumber
\end{eqnarray}

The introduction of the different operators $\calA_\calS$ leads to an analysis different from the one conducted in \cite{MalgouyresLandsberg_long}. Instead of considering the intersection of one linear space with a subset of $\Sigma_2$ (as in \cite{MalgouyresLandsberg_long}), we consider the intersection of many linear sets (the kernels of the operator $\calA_\calS$) with $\Sigma_1$.

The following property will turn out to be necessary and sufficient to guarantee the stable recovery property.
\begin{defi}{\bf Sparse-Deep-Null Space Property}

Let $\gamma\geq  1$ and $\rho>0$, we say that $\calA$ satisfies the {\em \SNSPlong (\SNSP)} with constants $(\gamma,\rho)$ for $\SF$ if and only if for all $\calS$ and $\calS'\in\SF$, any $T\in P(\wS_\calS) + P(\wS_{\calS'})$ satisfying $\|\calA_{\calS\cup\calS'} T\|\leq \rho$ and any $T'\in\KER{\calA_{\calS\cup\calS'}}$, we have
\begin{equation}\label{dsnsp}
\|T\| \leq \gamma \|T-\PROJ{\calS\cup\calS'} T'\|.
\end{equation}
\end{defi}
 Geometrically, the \SNSP does not hold when $\PROJ{\calS\cup\calS'} \KER{\calA_{\calS\cup\calS'}}$ intersects $P(\wS_\calS) + P(\wS_{\calS'})$ away from the origin or tangentially at $0$. It holds when the two sets intersect "transversally" at $0$.  Despite an apparent abstract nature, we will be able to characterize precisely when the lifting operator corresponding to a convolutional linear network satisfies the \SNSP (see Section \ref{conv-tree}). We will also be able to calculate the constants $(\gamma,\rho)$. 

%An interesting property of the \SNSP is that, as stated in the next proposition, it can be a consequence of simple tests.

\begin{prop}{\bf Sufficient condition for \SNSP}\label{suff_prop}

If $\KER{\calA}\cap \TS_{\calS\cup\calS'} = \{0\}$,  for all $\calS$ and $\calS'\in\SF$, then $\calA$ satisfies the \SNSP with constants $(\gamma,\rho)= (1,+\infty)$ for $\SF$.

%As a consequence, almost every $\calA$ such that
%\[\RH{\calA} \geq \max_{\calS, \calS'\in\SF}(\dim(\TS_{\calS\cup\calS'}))\]
%$\calA$ satisfies the \SNSP with constants $(\gamma,\rho)= (1,+\infty)$.
\end{prop}
\proof
In order to prove the proposition, let us consider $\calS$ and $\calS'\in\SF$, $T'\in\KER{\calA_{\calS\cup\calS'}}$. We have $ \calA\PROJ{\calS\cup\calS'}T' = 0 $ and therefore $ \PROJ{\calS\cup\calS'}T' \in \KER{\calA}$. Moreover, by definition, $ \PROJ{\calS\cup\calS'}T' \in \TS_{\calS\cup\calS'} $. Therefore, applying the hypothesis of the proposition, we obtain $\PROJ{\calS\cup\calS'}T'=0$ and \eqref{dsnsp} holds for any $T$, when $\gamma=1$. Therefore, $\calA$ satisfies the \SNSP with constants $(\gamma,\rho)= (1,+\infty)$  for $\SF$.

%The second assertion is a straightforward consequence of the first one.
\endproof
%Notice that if $\calA$ 
If $\iS \in \SF$, the condition becomes $\KER{\calA} = \{0\}$, which is sufficient but obviously not necessary for the \SNSP to hold. However, when $\SF$ truly imposes sparsity, the condition $\KER{\calA}\cap \TS_{\calS\cup\calS'} = \{0\}$ says that the elements of $\KER{\calA}$ shall not be sparse in some (tensorial) way. This nicely generalizes the case $H=1$.
%For instance, for any given $S'\leq\frac{S}{2}$, if $\SF$ contains all the supports $\calS$ such that, for all $k=1..H$, $|\calS_k| = S'$, we have $\max_{\calS, \calS'\in\SF}(\dim(\TS_{\calS\cup\calS'})) = (2 {S'})^H$. The proposition guarantees that almost every $\calA$ such that  $\RH{\calA} \geq (2 {S'})^H$ satisfies the \SNSP with constants $(\gamma,\rho)= (1,+\infty)$.

%Interestingly, in the compressed sensing framework when $H=1$ and $\calA$ is a sampling matrix, the first statement of the above proposition says that any sampling matrix with column rank larger than twice the maximal sparsity allowed by the model satisfies the \SNSP with constants $(\gamma,\rho)= (1,+\infty)$. To the best of our knowledge, this leads, even in the case $H=1$ to a new stability condition.

\begin{defi}{\bf Deep-lower-RIP constant}\label{rip-def}

There exists a constant $\sigma_{\SF} > 0$ such that for any $\calS$ and $\calS'\in\SF$ and any $T$ in the orthogonal complement of $\KER{\calA_{\calS\cup\calS'}}$
\begin{equation}\label{dlrip}
\sigma_{\SF} \|\PROJ{\calS\cup\calS'}T\| \leq \|\calA_{\calS\cup\calS'}T\|.
\end{equation}

We call $\sigma_{\SF}$ a Deep-lower-RIP constant of $\calA$ with regard to $\SF$.
\end{defi}
\proof 
The existence of $\sigma_{\SF}$ is a straightforward consequence of the fact that
the restriction of $\calA_{\calS\cup\calS'}$ on the orthogonal complement of $\KER{\calA_{\calS\cup\calS'}}$ is injective. We therefore have for all $T$ in the orthogonal complement of $\KER{\calA_{\calS\cup\calS'}}$
\[\|\calA_{\calS\cup\calS'}T\|\geq \sigma_{\calS\cup\calS'} \|T\|\geq \sigma_{\calS\cup\calS'} \|\PROJ{\calS\cup\calS'}T\|,
\]
where $\sigma_{\calS\cup\calS'}>0$ is the smallest non-zero singular value of $\calA_{\calS\cup\calS'}$. The last inequality holds because $\PROJ{\calS\cup\calS'}$ is a contraction.

We obtain the existence of $\sigma_{\SF}$ by taking the minimum of the constants $\sigma_{\calS\cup\calS'}$ over the finite family of $\calS$ and $\calS'\in\SF$.
\endproof

%We now state a sufficient condition for the stable recovery.

\begin{thm}{\bf Sufficient condition of stable recovery for structured linear networks}\label{suf-stble-recovery-sparse}

Consider a structured linear network defined by $M_1$, \ldots, $M_H$, sparsity constraints defined by a family of possible supports $\SF$, data $X$ and $Y$ and the operator $\calA$ satisfying \eqref{lifting}.

Assume  $\calA$ satisfies the \SNSP with the constants $\gamma\geq 1$, $\rho>0$ for $\SF$. For any $\overline\calS\in\SF$, $\overline\wbf \in \wS_{\overline\calS}$ and $\overline\calS'\in\SF$, $\overline\wbf' \in \wS_{\overline\calS'}$ as in \eqref{model_estim-sparse} with $\eta+\delta\leq \rho$, we have
\[ \|P(\overline\wbf') - P(\overline\wbf)\|\leq \frac{\gamma }{ \sigma_{\SF}}~(\delta+\eta),
\]
where  $\sigma_{\SF}$ is the Deep-lower-RIP constant of $\calA$ with regard to $\SF$. 

Moreover, if $\frac{\gamma }{ \sigma_{\SF}}~(\delta+\eta)\leq \frac{1}{2}  ~ \max\left(\|P(\overline\wbf')\|_\infty,\|P(\overline\wbf)\|_\infty\right) $, then
\[d_p ( [\overline\wbf'] , [\overline\wbf] ) \leq 7 (HS)^{\frac{1}{p}}     \min\left( \|P(\overline\wbf)\|_\infty^{\frac{1}{H}-1}, \|P(\overline\wbf')\|_\infty^{\frac{1}{H}-1} \right) \frac{\gamma}{\sigma_{\SF}}(\delta+\eta).
\]
\end{thm}
%The proof of the theorem is provided in Appendix \ref{suf-stble-recovery-sparse-proof}.
\begin{proof}
Because $\calA_{{\overline\calS}'\cup\overline{\calS}}$ is linear and then because $\overline\wbf \in \wS_{\overline\calS}$ and $\overline\wbf' \in \wS_{\overline\calS'}$, using \eqref{ASSpP}, we have
\begin{eqnarray}
\|\calA_{{\overline\calS}'\cup\overline{\calS}} ( P({\overline\wbf}') - P(\overline\wbf) )\| & = &\|\calA_{{\overline\calS}'\cup\overline{\calS}} P({\overline\wbf}') - \calA_{{\overline\calS}'\cup\overline{\calS}}P(\overline\wbf) \| \nonumber \\
& = &\|\calA P({\overline\wbf}') - \calA P(\overline\wbf) \|  \nonumber\\
& \leq & \|\calA P({\overline\wbf}') - X \| + \|\calA P(\overline\wbf) -X\| \nonumber\\
& \leq &  \delta + \eta \label{ietviqvu}
\end{eqnarray}

If we further decompose (the decomposition is unique)
\begin{equation}\label{noeqruinobt}
P({\overline\wbf}') - P(\overline\wbf) = T + T',
\end{equation}
where $T'\in\KER{\calA_{{\overline\calS}'\cup\overline{\calS}}}$ and $T$ is orthogonal to $\KER{\calA_{{\overline\calS}'\cup\overline{\calS}}}$, we have
\[\|\calA_{{\overline\calS}'\cup\overline{\calS}} (P({\overline\wbf}') - P(\overline\wbf) )\| = \|\calA_{{\overline\calS}'\cup\overline{\calS}}  T\| \geq \sigma_{\SF} \|\PROJ{{\overline\calS}'\cup\overline{\calS}}T\| ,
\]
where $\sigma_{\SF}$ is the Deep-lower-RIP constant of $\calA$ with regard to $\SF$. Combining with \eqref{ietviqvu}, we get
\[\|\PROJ{{\overline\calS}'\cup\overline{\calS}}T\| \leq \frac{ \delta + \eta}{\sigma_{\SF}}.
\]
Combining this inequality with $\PROJ{{\overline\calS}'\cup\overline{\calS}}P({\overline\wbf}') = P({\overline\wbf}')$, $\PROJ{{\overline\calS}'\cup\overline{\calS}}P(\overline\wbf) = P(\overline\wbf)$ and \eqref{noeqruinobt}, we obtain
%\begin{equation} \label{erunioqerbi}
\begin{eqnarray*}
\|P({\overline\wbf}') - P(\overline\wbf) - \PROJ{{\overline\calS}'\cup\overline{\calS}}T' \| &=&\| \PROJ{{\overline\calS}'\cup\overline{\calS}}\left(P({\overline\wbf}') - P(\overline\wbf) -T'\right) \|\\ 
 & = & \|\PROJ{{\overline\calS}'\cup\overline{\calS}}T\| \\
 &\leq  & \frac{\delta+\eta}{\sigma_{\SF}}.
\end{eqnarray*}
%\end{equation}

Combining the latter inequality with the hypotheses:$\calA$ satisfies the \SNSP with constants $(\gamma,\rho)$ for $\SF$ and  $ \delta+\eta \leq \rho$; we have
\begin{eqnarray*} 
\|P({\overline\wbf}') - P(\overline\wbf)\| & \leq & \gamma\|P({\overline\wbf}') - P(\overline\wbf) - \PROJ{{\overline\calS}'\cup\overline{\calS}}T' \|  \\
&\leq & \gamma ~\frac{\delta+\eta}{\sigma_{\SF}}.
\end{eqnarray*}
When $\delta+\eta$ satisfy the condition in the theorem, we can apply Theorem \ref{rk1-Ident-thm} and obtain the last inequality.
\end{proof}

Theorem \ref{suf-stble-recovery-sparse} differs from the analogous theorem in \cite{MalgouyresLandsberg_long}. In particular, it is dedicated to sparsity constraints. The constant of the upper bound is different. We replace the smallest non-zero singular value of an operator by the min, over a finite number of linear space, of the smallest non-zero singular value of the restriction of the operator on the linear space (see Definition \ref{rip-def} and its proof). This is the usual role of the lower-RIP constant in compressed sensing \cite{elad_book}, hence the name Deep-lower-RIP.

One might again ask whether the condition \lq\lq$\calA$ satisfies the \SNSP\rq\rq\  is sharp or not. As stated in the following theorem, the answer is affirmative.

\begin{thm}{\bf Necessary condition for stable recovery for structured linear networks}\label{nec-stble-recovery-sparse}

Consider a structured linear network defined by $M_1$, \ldots, $M_H$, , sparsity constraints defined by a family of possible supports $\SF$, data $X$ and $Y$ and the operator $\calA$ satisfying \eqref{lifting}.

Assume the stability property holds: There exists $C$ and $\delta>0$ such that for any $\overline\calS\in\SF$ and any $\overline{\wbf}\in\wS_{\overline\calS}$, any $Y=\calA P(\overline{\wbf}) + e$, with $\|e\| \leq \delta$, and any ${\overline\calS}'\in\SF$ and ${\overline\wbf}'\in\wS_{{\overline\calS}'}$
such that
\[\|\calA P({\overline\wbf}') -Y \| \leq \|e\|
\]
we have
\[d_2 ( [{\overline\wbf}'] , [\overline\wbf] ) \leq C ~ \min\left( \|P(\overline\wbf)\|_\infty^{\frac{1}{H}-1}, \|P({\overline\wbf}')\|_\infty^{\frac{1}{H}-1} \right)  \|e\|.
\]

Then, $\calA$ satisfies the \SNSP with constants 
$$\gamma =  C S^{\frac{H-1}{2}} \sqrt{H}~\sigma_{max}~\qquad \mbox{ and } \qquad \rho = \delta,$$
for $\SF$, where $\sigma_{max}$ is the spectral radius of $\calA$.
\end{thm}
%The proof is very similar to the proof of the Theorem 6, in \cite{MalgouyresLandsberg_long} and the proof of the analogous converse statement in \cite{cohen2009compressed}. It is provided in Appendix \ref{nec-stble-recovery-sparse-proof}.
\begin{proof}
Let $\overline\calS$ and $\overline\calS'\in\SF$. Let $\overline{\wbf}\in\wS_{\overline\calS}$ and $\overline{\wbf}'\in\wS_{\overline\calS'}$ be such that $\|\calA\left(P(\overline{\wbf})- P(\overline{\wbf}')\right)\|\leq \delta$. We have, using  \eqref{ASSpP},
\[\calA\left(P(\overline{\wbf})- P(\overline{\wbf}')\right) = \calA_{\overline\calS \cup \overline\calS'}\left(P(\overline{\wbf})- P(\overline{\wbf}')\right).
\]
Throughout the proof, we also consider $T'\in\KER{\calA_{\overline\calS\cup\overline\calS'}}$. We assume that $\|P(\overline{\wbf})\|_\infty \leq \|P(\overline{\wbf}')\|_\infty$. When it is not the case, the proof is analogue. We denote
\[Y=\calA P(\overline{\wbf}')\qquad\mbox{ and } \qquad e=\calA P(\overline{\wbf}')-\calA P(\overline{\wbf}).
\]
We have $Y=\calA P(\overline{\wbf})+e$ with $\|e\|\leq \delta$. Moreover, since
$\overline{\wbf}\in\wS_{\overline\calS}$, $\|e\| \leq \delta$ and since we obviously have $\|\calA P({\overline\wbf}') -Y \| \leq \|e\|$, the assumption that the stability property holds guaranties
\[d_2([\overline\wbf'] , [\overline\wbf])\leq  C \|P(\overline\wbf')\|_\infty^{\frac{1}{H}-1} \|e\|.
\]
Using \eqref{ASSpP} and the fact that $e=\calA_{\overline\calS\cup\overline\calS'} (P(\overline{\wbf})- P(\overline{\wbf}') )$, for any  $T'\in\KER{\calA_{\overline\calS\cup\overline\calS'}}$
\begin{eqnarray*}
\| e \| & = & \| \calA_{\overline\calS\cup\overline\calS'} ( P(\overline{\wbf})- P(\overline{\wbf}') - T' )\|, \\
 & \leq & \sigma_{max} \| \PROJ{\overline\calS\cup\overline\calS'}( P(\overline{\wbf})- P(\overline{\wbf}') - T') \|, \\
 & = &\sigma_{max} \|  P(\overline{\wbf})- P(\overline{\wbf}') - \PROJ{\overline\calS\cup\overline\calS'} T' \|,
\end{eqnarray*}
where $\sigma_{max}$ is the spectral radius of $\calA$. Therefore,
\[d_2([\overline\wbf'] , [\overline\wbf])\leq C \|P(\overline\wbf')\|_\infty^{\frac{1}{H}-1} ~\sigma_{max}~ \|P(\overline{\wbf})- P(\overline{\wbf}') - \PROJ{\overline\calS\cup\overline\calS'}T'  \|,
\]

Finally, using Theorem \ref{PLip-thm} and the fact that $\|P(\overline{\wbf})\|_\infty \leq \|P(\overline{\wbf}')\|_\infty$, we obtain
\begin{eqnarray*}
\|P(\overline\wbf')-P(\overline\wbf)\| & \leq & S^{\frac{H-1}{2}} H^{1-\frac{1}{2}}\|P(\overline\wbf')\|_{\infty}^{1-\frac{1}{H}} d_2([\overline\wbf'], [\overline\wbf]) \\
 & \leq & C S^{\frac{H-1}{2}} \sqrt{H} ~\sigma_{max}~\|P(\overline{\wbf})- P(\overline{\wbf}') - \PROJ{\overline\calS\cup\overline\calS'}T'  \| \\
 & = & \gamma \|P(\overline{\wbf})- P(\overline{\wbf}') - \PROJ{\overline\calS\cup\overline\calS'}T'  \|
\end{eqnarray*}
for $\gamma = C S^{\frac{H-1}{2}} \sqrt{H}~\sigma_{max}~$.

Summarizing, we conclude that under the hypothesis of the theorem: For any $\overline\calS$ and $\overline\calS'\in\SF$ and  any $T\in P(\wS_{\overline\calS}) + P(\wS_{\overline\calS'}) $ (above $P(\overline{\wbf})- P(\overline{\wbf}')$ has the role of $T$) such that $\|\calA T\|=\|\calA_{\overline\calS\cup\overline\calS'} T\|\leq \delta$, we have for any $T'\in\KER{\calA_{\overline\calS\cup\overline\calS'}}$
\[ \|T\|\leq \gamma \| T - \PROJ{\overline\calS\cup\overline\calS'}T' \|.
\]
In words, $\calA$ satisfies the \SNSP for $\SF$ with the constants of Theorem \ref{nec-stble-recovery-sparse}.
\end{proof}

%% file: example_spars1.tex
\section{Application to convolutional linear network under sparsity prior}\label{conv-tree}

%%%%%%%%%%%%%%%%%%%%%%%%%%%%%%%%%%%%%%%%%%%%%%%% 

\begin{figure}
\centering{
\begin{picture}(15,10)
%\put(11.8,11){\shortstack{$\RP$}}
% Commentaire sur le niveau des arc
\put(2.5,9){\shortstack{{\scriptsize input layer}}}
\put(8,9){\shortstack{{\scriptsize \green{hidden layers}}}}
\put(14,9){\shortstack{{\scriptsize \red{output layer}}}}

% feuilles

\put(4,2){\circle*{0.5}}
\put(4,2){\line(1,0.25){4}}
\put(4,4){\circle*{0.5}}
\put(4,4){\line(1,-0.25){4}}
\put(4,6){\circle*{0.5}}
\put(4,6){\line(1,0.25){4}}
\put(4,8){\circle*{0.5}}
\put(4,8){\line(1,-0.25){4}}

% niveau 2

\put(8,3){\line(1,0){4}}
\put(8,3){\line(1,1){4}}
\put(8,3){\green{{\circle*{0.5}}}}

\put(8,7){\line(1,0){4}}
\put(8,7){\line(1,-1){4}}
\put(8,7){\green{{\circle*{0.5}}}}
% niveau 1
\put(12,3){\line(1,0.5){4}}
\put(12,3){\green{{\circle*{0.5}}}}

\put(12,7){\line(1,-0.5){4}}
\put(12,7){\green{{\circle*{0.5}}}}

% racine

\put(16,5){\red{{\circle*{0.5}}}}

\end{picture}
}

\caption{\label{network}Example of a convolutional linear network. To every edge is attached a convolution kernel. The network does not involve non-linearities or sampling.}
\end{figure}
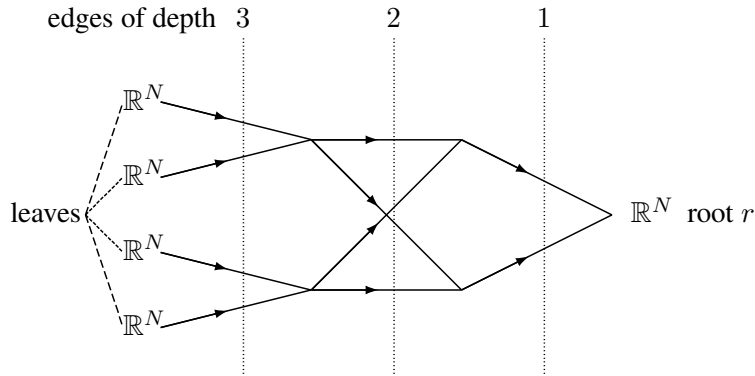

%%%%%%%%%%%%%%%%%%%%%%%%%%%%%%%%%%%%%%%%%%%%%%%%%%

We consider a sparse convolutional linear network as depicted in Figure \ref{network}. Formally, the considered convolutional linear network is defined from a rooted directed acyclic graph $\tree(\edges,\nodes)$ composed of nodes $\nodes$ and edges $\edges$. Each edge connects two nodes. The root of the graph is denoted by $r$ (it contains the output signal) and the set containing all its leaves is denoted by $\leaves$ (the leaves contain the input signal). We denote by $\PP$ the set of all paths connecting the leaves and the root. We assume, without loss of generality, that the length of any path between a leaf and the root is independent of the considered path and equal to $H\geq 0$. We also assume that, for any edge $e\in\edges$, the length of the paths separating $e$ and any leaf is constant. This length is called the depth of $e$. For any $h=1..H$, we denote the set containing all the edges of depth $h$, by $\edges(h)$.

Moreover, to any edge $e$ is attached a convolution kernel of maximal support $\calS_e\subset \N{N}$. We assume (without loss of generality) that $\sum_{e\in\edges(h)} |\calS_e|$ is independent of $h$ ($|\calS_e|$ denotes the cardinality of $\calS_e$). We take
\[S= \sum_{e\in\edges(1)} |\calS_e|.
\]
For any edge $e$, we consider the mapping $\calT_e: \RR^S \longrightarrow \RR^N$ that maps any $w\in\RR^S$ into the convolution kernel $\calT_e(w)\in\RR^N$, attached to the edge $e$, whose support is $\calS_e$. As in the previous section, we assume a sparsity constraint and will only consider a family $\Mod{}$ of possible supports $\calS \subset \iS$.

At each $h$, the convolutional linear network computes, for all $e\in\edges(h)$,  the convolution between the signal at the origin of $e$; then, it attaches to any ending node the sum of all the convolutions arriving at that node. Examples of such convolutional linear networks includes wavelets, wavelet packets \cite{Mallatbook} or the fast transforms optimized in \cite{FTL_IJCV,FTO}. It is the usual convolutional neural network, without bias, in which the activation function is the identity and the supports are potentially scattered and not fixed. It is clear that the operation performed between any pair of consecutive layers depends linearly on parameters $w\in\RR^S$. The convolutional linear network therefore depends on parameters $\wbf\in\wS$ and its prediction takes the form
\[f_\wbf(x) = M_H(\wbf_H)\cdots M_1(\wbf_1) x^{|\leaves|}\qquad \mbox{, for all }x\in\RR^N
\]
where the operators $M_h$ satisfy the hypothesis of the present paper and $ x^{|\leaves|} = \Idl x$ where $\Idl$ concatenates vertically $|\leaves|$ identity matrix of size $N\times N$:
\begin{equation}\label{Idl_def}
\Idl=\left(\begin{array}{c} \Id \\ \vdots \\ \Id\end{array}\right) \in \RR^{|\leaves| N \times N}.
\end{equation} 
Given a sample $(x_i,y_i)_{i=1..n} \in(\RR^N\times \RR^N)^n$ and reminding that $X$ is the horizontal concatenation of the column vectors $x_i$, we also denote $X^{|\leaves|} = \Idl X \in\RR^{N|\leaves| \times n}$.

Given $X^{|\leaves|}$ and a network architecture, this section applies the results of the preceding sections in order to  identify sharp conditions guaranteeing that, for  any supports $\overline\calS$ and  $\overline\calS'\in\Mod{}$, any parameters  $\overline{\wbf}$ and $\overline{\wbf}'\in\wS$ satisfying $\SUPP{\overline\wbf} \subset \overline\calS$ and  $\SUPP{\overline\wbf}' \subset \overline\calS'$, and such that
\[\| M_H(\overline\wbf_H)\cdots M_1(\overline\wbf_1) X^{|\leaves|} - Y\| = \delta \qquad \mbox{and}\qquad\| M_H(\overline\wbf'_H)\cdots M_1(\overline\wbf'_1) X^{|\leaves|} - Y\| = \eta
\]
are small enough, we can guarantee that $\overline\wbf$ and $\overline\wbf'$ are close to each other.

In order to do so, we first establish a few simple properties and define relevant notations. Notice first that, we can apply the convolutional linear network to any input $u\in\codeset$, where $u$ is the (vertical) concatenation of the signals $u^f\in\RR^N$ for $f\in\leaves$. Therefore, $M_H(\wbf_H)\cdots M_1(\wbf_1)$ is the (horizontal) concatenation of $|\leaves|$ matrices $Z^f\in\RR^{N\times N}$ such that
\begin{equation}\label{avecZ}
M_H(\wbf_H)\cdots M_1(\wbf_1) u = \sum_{f\in\leaves} Z^fu^f\qquad\mbox{, for all }u\in\codeset.
\end{equation}
Let us consider the convolutional linear network defined by $\wbf\in\wS$ as well as $f\in\leaves$ and $n=1..N$. The column of $M_H(\wbf_H)\cdots M_1(\wbf_1)$ corresponding to  the leaf $f$ and the entry $n$  is the translation by $n$ of
\begin{equation}\label{multiconv}
\sum_{\pbf\in\PP(f)} \multiconv{\pbf}{\wbf}
\end{equation}
where $\PP(f)$ contains all the paths of $\PP$ starting from the leaf $f$ and
\[\multiconv{\pbf}{\wbf} = \calT_{e^H}(\wbf_H)*\ldots*\calT_{e^1}(\wbf_1) \qquad\mbox{, where }\pbf=(e^1,\ldots,e^H)
\] 
and we remind that $\calT_{e^h}(\wbf_h)$ is the convolution kernel on the edge $e^h$.

We define for any $h=1..H$ the mapping $\ebf_h : \N{S} \longrightarrow \edges(h)$ which provides for any $i=1..S$ the unique edge of $\edges(h)$ such that the $i^{\mbox{th}}$ entry of $w\in\RR^S$ contributes to $\calT_{\ebf_h(i)}(w)$. Also, for any $\ibf\in\N{S}^H$, we denote $\pbf_\ibf=(\ebf_1(\ibf_1), \ldots, \ebf_H(\ibf_H))$ and, for any $\calS\in\Mod{}$,
\[\Ibf_\calS = \left\{\ibf \in\iS | \ibf \in \calS\mbox{ and }\pbf_\ibf \in \PP \right\}.
\] 
The latter contains all the indices of $\calS$ corresponding to a valid path in the network. For any set of parameters $\wbf\in\wS$ and any path $\pbf\in\PP$, we also denote by $\wbf^\pbf$  the restriction of $\wbf$ to its indices contributing to the kernels on the path $\pbf$.  We also define, for any $\ibf\in\iS$, $\wbf^\ibf\in\wS$ by 
\begin{equation}\label{hibf}
\wbf^\ibf_{h,j} = \left\{\begin{array}{ll}
1 & \mbox{, if } j =\ibf_h \\
0 & \mbox{otherwise}
\end{array}\right. \qquad\mbox{, for all }h=1..H\mbox{ and }  j=1..S
\end{equation}
so-that $P(\wbf^\ibf)$ is a Dirac at position $\ibf$. The difference between $\wbf^\pbf$ and $\wbf^\ibf$ will not be ambiguous, once in context.

We can deduce from \eqref{multiconv} that, when $\ibf\in\Ibf_\calS$,  $M_H(\wbf^\ibf_H)\cdots M_1(\wbf^\ibf_1)$ simply convolves the entries at one leaf with a Dirac delta function. Therefore, all  the entries of $M_H(\wbf^\ibf_H)\cdots M_1(\wbf^\ibf_1)$ are in $\{0,1\}$ and we denote $\calD_\ibf = \{(i,j)\in\N{N}\times\N{N|\leaves|} | \big(M_H(\wbf^\ibf_H)\cdots M_1(\wbf^\ibf_1)\big)_{i,j} = 1\}$. 

We also denote $\one\in\RR^S$   a vector of size $S$ with all its entries equal to $1$. For any edge $e\in\edges$, $\one^e\in\RR^S$ consists of zeroes except for the entries contributing to the convolution kernel on the edge $e$ which are equal to $1$. For any $\calS\subset \iS$, we define $\one^\calS\in\wS$ which consists of zeroes except for the entries corresponding to the indexes in $\calS$  which are equal to $1$.

The equivalence relationship $\sim$, defined in Section \ref{notation-sec}, does not suffice to group parameters leading to the same network prediction. Indeed, with the considered convolutional networks, we can rescale the kernels on different path differently.
Therefore, we say that two networks sharing the same architecture and defined by the parameters $\wbf$ and $\wbf'\in\wS$ are equivalent if and only if
\[\forall \pbf\in\PP, \exists (\lambda_e)_{e\in\pbf} \in\RR^\pbf\mbox{, such that } \prod_{e\in\pbf} \lambda_e = 1 \mbox{ and } \forall e\in\pbf, \calT_e(\wbf') = \lambda_e \calT_e(\wbf). 
\]
The equivalence class of $\wbf\in\wS$ is denoted by $\class{\wbf}$. It is not difficult to see that the prediction of the networks defined by equivalent parameters are identical.  For any $p\in[1,+\infty[$, we define
\begin{equation}\label{Delta_def}
\cald_p (\class{\wbf}, \class{\wbf'}) = \Big(\sum_{\pbf\in\PP} d_p\big([\wbf^\pbf],[{\wbf'}^\pbf]\big)^p  \Big)^{\frac{1}{p}},
\end{equation}
where we remind that $d_p$ is defined in \eqref{dp_def}. Since $d_p$ is a metric, $\cald_p$ is a metric between network classes.

The equivalence classes we have defined do not take into the account the fact it is possible to modify $\wbf$ in a way that corresponds to permutation of the nodes of the network. Taking into account this invariant is difficult and remains an open question. It has not been addressed in \cite{arora2014provable,brutzkus2017globally,li2017convergence,sedghi2014provable,zhong2017recovery,MalgouyresLandsbergITW,MalgouyresLandsberg_long}.

Finally, we remind that because of \eqref{lifting}, there exists a unique mapping 
\[\calA:\TS\longrightarrow \RR^{N\times n}\] 
such that 
\[\calA P(\wbf) = M_H(\wbf_H)\cdots M_1(\wbf_1)X^{|\leaves|}\qquad \mbox{, for all } \wbf\in\wS,
\]
where $P$ is the Segre embedding defined in \eqref{defP}. 

\begin{prop}{\bf Necessary condition of identifiability of a sparse network}\label{nec-ident-network}

Only one of the two following alternatives can occur.
\begin{enumerate}
\item \label{tbhhteb} Either, there exist $\calS$ and $\calS'\in\Mod{}$ such that some entries of $M_H(\one_H^{\calS\cup\calS'})\cdots M_1(\one_1^{\calS\cup\calS'})\Idl$ do not belong to $\{0,1\}$. 

When this holds, $\class{\overline \wbf}$ is not always identifiable: there exists $\{\overline\wbf\} \neq \{\overline\wbf'\}$ such that  
\[ M_H(\overline\wbf_H)\cdots M_1(\overline\wbf_1) \Idl=  M_H(\overline\wbf'_H)\cdots M_1(\overline\wbf'_1)\Idl.
\]
\item \label{eohur} Or, for any $\calS$ and $\calS'\in\Mod{}$, all the entries of $M_H(\one_H^{\calS\cup\calS'})\cdots M_1(\one_1^{\calS\cup\calS'})\Idl$ belong to $\{0,1\}$. When this holds :
\begin{enumerate}
\item \label{second_item} For any $\calS$ and $\calS'\in\Mod{}$, for any distinct $\ibf\in\calS$ and $\ibf'\in\calS'$, we have $\calD_\ibf \cap \calD_{\ibf'}  = \emptyset$. 
\item \label{premier_item} For any $\calS$ and $\calS'\in\Mod{}$, for any $\wbf\in\wS_{\calS}$ and  $\wbf'\in\wS_{\calS'}$ and any distinct $\pbf$ and $\pbf'\in\PP$, we have 
\[\supp \Big( M_H(\wbf_H^{\pbf})\cdots M_1(\wbf_1^{\pbf}) \Idl\Big) \bigcap \SUPP{ M_H((\wbf')_H^{\pbf'})\cdots M_1((\wbf')_1^{\pbf'}) \Idl} = \emptyset. 
\]
\item \label{epruhu} If moreover $|\PP|=1$ and $X$ is full row rank: 
\[\KER{\calA_{\calS\cup\calS'}} = \{T\in\TS | \forall \ibf\in\Ibf_{\calS\cup\calS'}, T_\ibf = 0\}.
\]
\end{enumerate}
\end{enumerate}

\end{prop}

The proof is in Appendix \ref{nec-ident-network-proof}.
%Proposition \ref{nec-ident-network} extends Proposition 8 of \cite{MalgouyresLandsberg_long} by considering several possible supports. Said differently, Proposition 8 of \cite{MalgouyresLandsberg_long} corresponds to  Proposition \ref{nec-ident-network} when $\Mod{} = \{\iS\}$.

Proposition  \ref{nec-ident-network}, Item \ref{tbhhteb}, expresses a necessary condition of stability: for any $\calS$ and $\calS'\in\Mod{}$, all the entries of $M_H(\one_H^{\calS\cup\calS'})\cdots M_1(\one_1^{\calS\cup\calS'})\Idl$ belong to $\{0,1\}$. The condition is restrictive but not empty. We will see in the sequel that, when $X$ is full row rank, the condition is sufficient to guarantee the stability. Notice that the condition can be computed at a low cost by applying the network to Dirac delta functions, when $|\Mod{}|$ is not too large.

\begin{prop}\label{network-cor}
If $|\PP|=1$ and $X$ is full row rank. If, for any $\calS$ and $\calS'\in\SF$, all the entries of $M_H(\one_H^{\calS\cup\calS'})\cdots M_1(\one_1^{\calS\cup\calS'})M_0$ belong to $\{0,1\}$, then $\KER{\calA_{\calS\cup\calS'}}$ is the orthogonal complement of $\TS_{\calS\cup\calS'}$ and  $\calA$ satisfies the \SNSP with constants $(\gamma,\rho)=(1,+\infty)$ for $\SF$. Moreover,  $\sigma_{\SF} = \sqrt{N} \sigma_{min}(X)$, where $ \sigma_{min}(X)$ is the smallest singular value of $X$, is a deep-lower-RIP constant of $\calA$ with regard to $\SF$.
\end{prop}

The proof of the proposition is in Appendix \ref{network-cor-proof}.

Let us remind: we consider $X$ and $Y\in\RR^{N\times n}$, $\overline{\calS}$ and $\overline{\calS}'\in\Mod{}$ and parameters $\overline \wbf$ and $\overline \wbf' \in\wS$ satisfying  
\begin{equation}\label{etibte}
\SUPP{\overline{\wbf}} \subset \overline{\calS} \qquad \mbox{and}\qquad \SUPP{\overline{\wbf}'} \subset \overline{\calS}'
\end{equation}
and denote
\begin{equation}\label{iuhqribbrv} 
\delta =\| M_H(\overline\wbf_H)\cdots M_1(\overline\wbf_1)X^{|\leaves|} - Y\| \qquad\mbox{and}\qquad  \eta = \| M_H(\overline\wbf'_H)\cdots M_1(\overline\wbf'_1)X^{|\leaves|} - Y\|, 
\end{equation}
where we will assume in the theorem that $\delta$ and $\eta$ are small.

 For any path $\pbf \in\PP$, we denote 
\[\delta^\pbf = \| M_H(\overline\wbf^\pbf_H)\cdots M_1(\overline\wbf^\pbf_1) - M_H((\overline\wbf')^\pbf_H)\cdots M_1((\overline\wbf')^\pbf_1) \|
\]
where we remind that $\overline\wbf^\pbf$ (resp ${\overline\wbf'}^\pbf$) denotes the restriction of $\overline\wbf$ (resp $\overline\wbf'$) to the path $\pbf$. Under the hypothesis of the following theorem, when $\delta + \eta$ is small, we will prove that $\delta^\pbf$ is small too for every $\pbf\in\PP$ (see \eqref{oreibtbtsg}).

\begin{thm}{\bf Sufficient condition of stability}\label{stabl-rec-network-thm}

Let $X$, $Y$, $\overline\calS$, $\overline\calS'$, $\overline\wbf$, $\overline\wbf'$, $\delta$ and $\eta$ be as described above (see \eqref{etibte} and \eqref{iuhqribbrv}). Assume $X$ is full row rank.

If for any $\calS$ and $\calS'\in\SF$, all the entries of $M_H(\one_H^{\calS\cup\calS'})\cdots M_1(\one_1^{\calS\cup\calS'})\Idl$ belong to $\{0,1\}$ and if there exists $\varepsilon >0$ such that for all $e\in\edges$, $\|\calT_e(\overline\wbf)\|_\infty\geq \varepsilon$ and for all $\pbf\in\PP$, $\frac{\delta^\pbf}{\sqrt{N} \sigma_{min}(X)}\leq \frac{1}{2} \max(\|P(\overline \wbf^\pbf)\|_\infty,\|P((\overline \wbf')^\pbf)\|_\infty)$, then $\overline \wbf$ and $\overline \wbf'$ are close to each other: for any $p\in [1,\infty[$
\[\cald_p (\class{{\overline\wbf}'}, \class{\overline\wbf}) \leq 7 \frac{(HS)^{\frac{1}{p}}}{\sqrt{N} \sigma_{min}(X)^2 \varepsilon^{H-1}} ~ (\delta+\eta).
\]
\end{thm}
We remind that, according to Proposition \ref{nec-ident-network}, Item \ref{tbhhteb},  the network is not identifiable when some entries of $M_H(\one_H^{\calS\cup\calS'})\cdots M_1(\one_1^{\calS\cup\calS'})\Idl$ do not belong to $\{0,1\}$. 

The proof of the theorem is in Appendix \ref{stabl-rec-network-thm-proof}.

\section{Conclusion}
We provide a necessary and sufficient condition of stability for the optimal weights of a sparse linear network. In the general setting, when no assumption is made on the architecture of the network, the stability constant $C$ is improved when compared to un-specified weight models \cite{MalgouyresLandsberg_long}. The gain is comparable to the gain obtained in compressed sensing when replacing the smallest singular value by the lower RIP constant \cite{elad_book}. We then specialize the results to sparse convolutional linear networks. In this analyses, we detail the stability condition in terms of a condition on the architecture and a condition on the sample inputs. The condition on the architecture is restrictive but not empty. The condition on the sample inputs is rather weak and basically requires to have as many (diverse) samples as the dimension of the input space. The constant $\sigma_{min}(X)$ is a key component of the stability constant.

%% file: appendices.tex
\appendix

\section{Proof of Proposition \ref{nec-ident-network}}\label{nec-ident-network-proof}

First notice that the entries of $M_H(\one_H^{\calS\cup\calS'})\cdots M_1(\one_1^{\calS\cup\calS'})$ are non-negative integers.

{\em Let us first assume that: } There exist $\calS$ and $\calS'\in\Mod{}$ and an entry of \[M_H(\one_H^{\calS\cup\calS'})\cdots M_1(\one_1^{\calS\cup\calS'})\Idl\] that does not belong to $\{0,1\}$.

Using \eqref{Idl_def}, \eqref{avecZ} and \eqref{multiconv}, we know that there is $n=1..N$ such that
\[\sum_{f\in\leaves} \sum_{\pbf\in\PP(f)} \multiconv{\pbf}{\one^{\calS\cup\calS'}}_n \geq 2.
\]
As a consequence, there is $\ibf$ and $\jbf\in{\calS\cup\calS'}$ with $\ibf\neq \jbf$ and
\[\multiconv{\pbf_\ibf}{\wbf^\ibf}_n = \multiconv{\pbf_\jbf}{\wbf^\jbf}_n = 1.
\]
Therefore, since both $\multiconv{\pbf_\ibf}{\wbf^\ibf}$ and $\multiconv{\pbf_\jbf}{\wbf^\jbf}$ are Diracs,
\[M_H(\wbf^\ibf_H)\cdots M_1(\wbf^\ibf_1) \Idl= M_H(\wbf^\jbf_H)\cdots M_1(\wbf^\jbf_1)\Idl.
\]
Since $\ibf\neq\jbf$, $\{\wbf^\ibf\} \neq \{\wbf^\jbf\}$ and the network is not identifiable. This proves Item \ref{tbhhteb}.

{\em Let us now assume that: } For any $\calS$ and $\calS'\in\SF$, all the entries of \[M_H(\one_H^{\calS\cup\calS'})\cdots M_1(\one_1^{\calS\cup\calS'})\Idl\] belong to $\{0,1\}$.

For any $\calS$ and $\calS'\in\SF$ and any distinct $\ibf\in\calS$ and $\ibf'\in\calS'$,  since $\multiconv{\pbf}{\wbf^\ibf}$ and $\multiconv{\pbf}{\wbf^{\ibf'}}$ are Diracs, using \eqref{multiconv}, \eqref{avecZ} and the hypothesis we establish Item \ref{second_item}.

To prove Item \ref{premier_item}, we consider $\calS$ and $\calS'\in\Mod{}$, $\wbf\in\wS_{\calS}$ and  $\wbf'\in\wS_{\calS'}$, and  distinct $\pbf \neq\pbf'\in\PP$. We have
\[\SUPP{M_H(\wbf_H^{\pbf})\cdots M_1(\wbf_1^{\pbf}) \Idl} \subset \SUPP{M_H((\one^{\calS\cup\calS'})_H^{\pbf})\cdots M_1((\one^{\calS\cup\calS'})_1^{\pbf})\Idl}
\]
and
\[\SUPP{M_H((\wbf')_H^{\pbf'})\cdots M_1((\wbf')_1^{\pbf'})\Idl} \subset \SUPP{M_H((\one^{\calS\cup\calS'})_H^{\pbf'})\cdots M_1((\one^{\calS\cup\calS'})_1^{\pbf'})\Idl}.
\]
Using the hypothesis, we know (as in the proof of Item \ref{second_item}) that
\begin{multline*}
\SUPP{M_H((\one^{\calS\cup\calS'})_H^{\pbf})\cdots M_1((\one^{\calS\cup\calS'})_1^{\pbf}) \Idl} \\
\bigcap 
\SUPP{M_H((\one^{\calS\cup\calS'})_H^{\pbf'})\cdots M_1((\one^{\calS\cup\calS'})_1^{\pbf'})\Idl} = \emptyset
\end{multline*}
and conclude that Item \ref{premier_item} holds.

To prove the Item \ref{epruhu}, notice first that $(P(\wbf^\ibf))_{\ibf\not\in\Ibf_{\calS\cup\calS'}}$ forms a basis of $\{T\in\TS | \forall \ibf\in\Ibf_{\calS\cup\calS'}, T_\ibf = 0\}$. We check using \eqref{multiconv} and \eqref{defAS} that, for any $\ibf\not\in\Ibf_{\calS\cup\calS'}$, 
\[\calA_{\calS\cup\calS'} P(\wbf^\ibf) = \left\{\begin{array}{ll} 
\calA 0 = 0 & \mbox{, if } \ibf\not\in{\calS\cup\calS'}\\
M_H(\wbf^\ibf_H)\cdots M_1(\wbf^\ibf_1) X^{|\leaves|} = 0 & \mbox{, if } \ibf\in{\calS\cup\calS'} \mbox{ and }\pbf_\ibf \not\in\PP.
\end{array}\right.
\]
As a consequence,  
\begin{equation}\label{reitb}
\{T\in\TS | \forall \ibf\in\Ibf_{\calS\cup\calS'}, T_\ibf = 0\} \subset \KER{\calA_{\calS\cup\calS'}}.
\end{equation}

To prove the converse inclusion, we observe that
\begin{eqnarray*}
\RH{\calA_{\calS\cup\calS'}} & = & \DIM{ \SPAN{\calA_{\calS\cup\calS'} P(\wbf^\ibf) | \ibf \in \Ibf_{\calS\cup\calS'}} } \\
 & = & \DIM{\SPAN{M_H(\wbf^\ibf_H)\cdots M_1(\wbf^\ibf_1) X^{|\leaves|} | \ibf \in \Ibf_{\calS\cup\calS'}}} \\
 & =  & \DIM{\SPAN{M_H(\wbf^\ibf_H)\cdots M_1(\wbf^\ibf_1)  | \ibf \in \Ibf_{\calS\cup\calS'}}}
\end{eqnarray*}
where the last equality holds because, when $|\PP|=1$, $X^{|\leaves|}=X$ is full row rank. Moreover, under the hypothesis of the proposition, for any distinct $\ibf$ and $\jbf\in\Ibf_{\calS\cup\calS'}$,  $\calD_\ibf \cap \calD_\jbf = \emptyset$, and therefore  
\[\DIM{\SPAN{M_H(\wbf^\ibf_H)\cdots M_1(\wbf^\ibf_1)  | \ibf \in \Ibf_{\calS\cup\calS'}}} = |\Ibf_{\calS\cup\calS'}|.\] 
Therefore, $\RH{\calA_{\calS\cup\calS'}} = |\Ibf_{\calS\cup\calS'}|$; i.e.
\[ S^H -\dim(\KER{\calA_{\calS\cup\calS'}}) = S^H - \dim(\{T\in\TS | \forall \ibf\in\Ibf_{\calS\cup\calS'}, T_\ibf = 0\})
\]
and $\dim(\KER{\calA_{\calS\cup\calS'}}) = \dim(\{T\in\TS | \forall \ibf\in\Ibf_{\calS\cup\calS'}, T_\ibf = 0\})$. Combined with \eqref{reitb}, we obtain
\[\KER{\calA_{\calS\cup\calS'}} = \{T\in\TS | \forall \ibf\in\Ibf_{\calS\cup\calS'}, T_\ibf = 0\}.
\]
This proves Item \ref{epruhu}.

\section{Proof of Proposition \ref{network-cor}}\label{network-cor-proof}

The fact that, $\KER{\calA_{\calS\cup\calS'}}$ is the orthogonal complement of $\TS_{\calS\cup\calS'}$ is a direct consequence of Proposition \ref{nec-ident-network}, Item \ref{epruhu}, and the fact that, when $|\PP|=1$, $\Ibf_{\calS\cup\calS'}=\calS\cup\calS'$. We then deduce that, for any $T'\in \KER{\calA_{\calS\cup\calS'}}$, $\PROJ{\calS\cup\calS'} T' = 0$. A straightforward consequence (see \eqref{dsnsp}) is that $\calA$ satisfies the \SNSP with constants $(\gamma,\rho)=(1,+\infty)$ for $\SF$. 

To calculate $\sigma_{\SF}$, let us consider $\calS$, $\calS'\in\SF$ and $T$ in the orthogonal complement of $\KER{\calA_{\calS\cup\calS'}}$. Using Proposition \ref{nec-ident-network}, Item \ref{epruhu}, we express $T$ under the form $T=\sum_{\ibf \in\calS\cup\calS'} T_\ibf P(\wbf^\ibf)$, where $\wbf^\ibf$ is defined by \eqref{hibf}. Using \eqref{defAS} and \eqref{proj_def}, the linearity of $\calA$ and the fact that, when $|\PP|=1$, $X^{|\leaves|}=X$, we obtain

\begin{eqnarray}
\|\calA_{\calS\cup\calS'} T \|^2 & = & \|\sum_{\ibf \in\Ibf} T_\ibf \calA P(\wbf^\ibf)  \|^2, \nonumber\\
 & = &  \|\sum_{\ibf \in\Ibf} T_\ibf M_H(\wbf^\ibf)\cdots M_1(\wbf^\ibf) X \|^2, \nonumber\\
 & \geq & \sigma^2_{min}(X)  \|\sum_{\ibf \in\Ibf} T_\ibf M_H(\wbf^\ibf)\cdots M_1(\wbf^\ibf) \|^2 \label{tibvt} 
\end{eqnarray}
Let us remind that, applying Proposition \ref{nec-ident-network}, Item \ref{second_item}, the supports of $M_H(\wbf^\ibf)\cdots M_1(\wbf^\ibf)$ (i.e. $\calD_\ibf$) and $M_H(\wbf^\jbf)\cdots M_1(\wbf^\jbf)$ (i.e. $\calD_\jbf$) are disjoint, when $\ibf\neq\jbf$. Let us also add that, since $\calA P(\wbf^\ibf)$ is the matrix of a convolution with a Dirac mass, we have $|\calD_\ibf| = N$, for all $\ibf\in\Ibf$. Combining these two properties with \eqref{tibvt} and reminding that $\|.\|$ is the Frobinius norm, we obtain
\begin{eqnarray*}
\|\calA_{\calS\cup\calS'} T \|^2 & \geq &  \sigma^2_{min}(X)  \sum_{\ibf \in\Ibf} T^2_\ibf \| M_H(\wbf^\ibf)\cdots M_1(\wbf^\ibf) \|^2 \\
 & = &\sigma^2_{min}(X)  N \sum_{\ibf \in\Ibf} T_\ibf^2 = \sigma^2_{min}(X)  N \|T\|^2.
\end{eqnarray*}
Using that $P_{\calS\cup\calS'}T = T$, we deduce the value of $\sigma_{\SF}$ in the proposition.

\section{Proof of Theorem \ref{stabl-rec-network-thm}}\label{stabl-rec-network-thm-proof}

Let us consider a path $\pbf\in\PP$, using \eqref{multiconv}, since all the entries of $M_H(\one_H^{\calS\cup\calS'})\cdots M_1(\one_1^{\calS\cup\calS'})\Idl$ belong to $\{0,1\}$, the restriction of the network to $\pbf$ satisfy the same property. Therefore, we can apply Proposition \ref{network-cor} and Theorem \ref{suf-stble-recovery-sparse} to the restriction of the convolutional linear network to $\pbf$, with
\[X'=Id \qquad\mbox{ and }\qquad Y'= M_H((\overline\wbf')^\pbf_H)\cdots M_1((\overline\wbf')^\pbf_1)
\]
and obtain, when $\frac{\delta^\pbf}{\sqrt{N}\sigma_{min}(X)}\leq \frac{1}{2} \max(\|P(\overline \wbf^\pbf)\|_\infty,\|P((\overline \wbf')^\pbf)\|_\infty)$, for any $p\in [1,\infty[$
%\[d_p ( [({\overline\wbf}')^\pbf] , [\overline\wbf^\pbf] ) \leq \frac{7 (HS)^{\frac{1}{p}}}{\sqrt{N}}     \min\left( \|P(\overline\wbf^\pbf)\|_\infty^{\frac{1}{H}-1}, \|P(({\overline\wbf}')^\pbf)\|_\infty^{\frac{1}{H}-1} \right) \delta^\pbf,
%\] 
\begin{equation}\label{tuoouqtnb}
d_p ( [({\overline\wbf}')^\pbf] , [\overline\wbf^\pbf] ) \leq 7 \frac{(HS)^{\frac{1}{p}}}{\sqrt{N} \sigma_{min}(X)} \varepsilon^{1-H} \delta^\pbf.
\end{equation} %We therefore have
We also have, using the definition of $X^{|\leaves|}$,
\begin{eqnarray*}
\delta + \eta   & = & \| M_H(\overline\wbf_H)\cdots M_1(\overline\wbf_1)X^{|\leaves|} - Y\| +  \| M_H(\overline\wbf'_H)\cdots M_1(\overline\wbf'_1)X^{|\leaves|} - Y\| \\
& \geq &  \| M_H(\overline\wbf_H)\cdots M_1(\overline\wbf_1)\Idl X - M_H(\overline\wbf'_H)\cdots M_1(\overline\wbf'_1)\Idl X \| \\
&  \geq & \sigma_{min}(X)  \|M_H(\overline\wbf_H)\cdots M_1(\overline\wbf_1)\Idl - M_H(\overline\wbf'_H)\cdots M_1(\overline\wbf'_1)  \Idl\| \\
& = &\sigma_{min}(X)  \sum_{\pbf \in\PP} \| M_H(\overline\wbf^\pbf_H)\cdots M_1(\overline\wbf^\pbf_1) \Idl - M_H((\overline\wbf')^\pbf_H)\cdots M_1((\overline\wbf')^\pbf_1) \Idl \| \\
& = &\sigma_{min}(X)  \sum_{\pbf \in\PP} \delta^\pbf
\end{eqnarray*}
where the penultimate equality is due to Proposition \ref{nec-ident-network}, Item  \ref{premier_item}. Combining this inequality, the definition of $\delta^\pbf$ and a standard norm inequality, we obtain
\begin{equation} \label{oreibtbtsg}
\left(\sum_{\pbf\in \PP} (\delta^\pbf)^p \right)^{\frac{1}{p}} \leq  \sum_{\pbf \in\PP} \delta^\pbf \leq \frac{\delta + \eta }{ \sigma_{min}(X)  }.
\end{equation}

Finally, combining the definition of the metric $\cald_p$ \eqref{Delta_def}, \eqref{tuoouqtnb} and the above inequality we obtain
\begin{eqnarray*}
\cald_p (\class{{\overline\wbf}'}, \class{\overline\wbf}) & \leq &7  \frac{(HS)^{\frac{1}{p}} }{\sqrt{N} \sigma_{min}(X)}  \varepsilon^{1-H} \left(\sum_{\pbf\in \PP} (\delta^\pbf)^p \right)^{\frac{1}{p}}, \\
& \leq & 7  \frac{(HS)^{\frac{1}{p}} }{\sqrt{N} \sigma_{min}(X)^2}  \varepsilon^{1-H} ~(\delta+\eta).\\
\end{eqnarray*}

%% file: sparse_net.bbl
\begin{thebibliography}{35}
\providecommand{\natexlab}[1]{#1}
\providecommand{\url}[1]{\texttt{#1}}
\expandafter\ifx\csname urlstyle\endcsname\relax
  \providecommand{\doi}[1]{doi: #1}\else
  \providecommand{\doi}{doi: \begingroup \urlstyle{rm}\Url}\fi

\bibitem[Ahmed et~al.(2014)Ahmed, Recht, and Romberg]{ahmed2014blind}
Arif Ahmed, Benjamin Recht, and Justin Romberg.
\newblock Blind deconvolution using convex programming.
\newblock \emph{IEEE Transactions on Information Theory}, 60\penalty0
  (3):\penalty0 1711--1732, 2014.

\bibitem[Arora et~al.(2012)Arora, Ge, Kannan, and Moitra]{arora2012computing}
Sanjeev Arora, Rong Ge, Ravindran Kannan, and Ankur Moitra.
\newblock Computing a nonnegative matrix factorization--provably.
\newblock In \emph{Proceedings of the forty-fourth annual ACM symposium on
  Theory of computing}, pages 145--162. ACM, 2012.

\bibitem[Arora et~al.(2014)Arora, Bhaskara, Ge, and Ma]{arora2014provable}
Sanjeev Arora, Aditya Bhaskara, Rong Ge, and Tengyu Ma.
\newblock Provable bounds for learning some deep representations.
\newblock In \emph{ICML}, pages 584--592, 2014.

\bibitem[Baldi and Hornik(1989)]{baldi1989neural}
Pierre Baldi and Kurt Hornik.
\newblock Neural networks and principal component analysis: Learning from
  examples without local minima.
\newblock \emph{Neural networks}, 2\penalty0 (1):\penalty0 53--58, 1989.

\bibitem[Bo\"{o}lcskei et~al.(2019)Bo\"{o}lcskei, Grohs, Kutyniok, and
  Petersen]{boolcskei2019optimal}
Helmut Bo\"{o}lcskei, Philipp Grohs, Gitta Kutyniok, and Philipp Petersen.
\newblock Optimal approximation with sparsely connected deep neural networks.
\newblock \emph{SIAM Journal on Mathematics of Data Science}, 1\penalty0
  (1):\penalty0 8--45, 2019.

\bibitem[Brutzkus and Globerson(2017)]{brutzkus2017globally}
Alon Brutzkus and Amir Globerson.
\newblock Globally optimal gradient descent for a convnet with gaussian inputs.
\newblock In \emph{Proceedings of the 34th International Conference on Machine
  Learning-Volume 70}, pages 605--614. JMLR. org, 2017.

\bibitem[Candes et~al.(2013)Candes, Strohmer, and
  Voroninski]{candes2013phaselift}
Emmanuel~J Candes, Thomas Strohmer, and Vladislav Voroninski.
\newblock Phaselift: Exact and stable signal recovery from magnitude
  measurements via convex programming.
\newblock \emph{Communications on Pure and Applied Mathematics}, 66\penalty0
  (8):\penalty0 1241--1274, 2013.

\bibitem[Chabiron et~al.(2014)Chabiron, Malgouyres, Tourneret, and
  Dobigeon]{FTL_IJCV}
Olivier Chabiron, Fran\c{c}ois Malgouyres, Jean-Yves Tourneret, and Nicolas
  Dobigeon.
\newblock Toward fast transform learning.
\newblock \emph{International Journal of Computer Vision}, pages 1--22, 2014.

\bibitem[Chabiron et~al.(2016)Chabiron, Malgouyres, Wendt, and Tourneret]{FTO}
Olivier Chabiron, Fran\c{c}ois Malgouyres, Herwig Wendt, and Jean-Yves
  Tourneret.
\newblock Optimization of a fast transform structured as a convolutional tree.
\newblock \emph{preprint HAL}, \penalty0 (hal-01258514), 2016.

\bibitem[Choromanska et~al.(2015{\natexlab{a}})Choromanska, Henaff, Mathieu,
  Arous, and LeCun]{choromanska2015loss}
Anna Choromanska, Mikael Henaff, Michael Mathieu, G{\'e}rard~Ben Arous, and
  Yann LeCun.
\newblock The loss surfaces of multilayer networks.
\newblock In \emph{Artificial Intelligence and Statistics}, pages 192--204,
  2015{\natexlab{a}}.

\bibitem[Choromanska et~al.(2015{\natexlab{b}})Choromanska, LeCun, and
  Arous]{choromanska2015open}
Anna Choromanska, Yann LeCun, and G{\'e}rard~Ben Arous.
\newblock Open problem: The landscape of the loss surfaces of multilayer
  networks.
\newblock In \emph{Conference on Learning Theory}, pages 1756--1760,
  2015{\natexlab{b}}.

\bibitem[Choudhary and Mitra(2014)]{choudhary2014identifiability}
Sunav Choudhary and Urbashi Mitra.
\newblock Identifiability scaling laws in bilinear inverse problems.
\newblock \emph{arXiv preprint arXiv:1402.2637}, 2014.

\bibitem[Donoho and Stodden(2003)]{donoho2003does}
David~L Donoho and Victoria Stodden.
\newblock When does non-negative matrix factorization give a correct
  decomposition into parts?
\newblock 2003.

\bibitem[Elad(2010)]{elad_book}
Michael Elad.
\newblock \emph{Sparse and Redundant Representations: From Theory to
  Applications in Signal and Image Processing}.
\newblock Springer, 2010.

\bibitem[Gribonval et~al.(2019)Gribonval, Kutyniok, Nielsen, and
  Voigtlaender]{gribonval:hal-02117139}
R{\'e}mi Gribonval, Gitta Kutyniok, Morten Nielsen, and Felix Voigtlaender.
\newblock {Approximation spaces of deep neural networks}.
\newblock working paper or preprint, June 2019.
\newblock URL \url{https://hal.inria.fr/hal-02117139}.

\bibitem[G{\"u}hring et~al.(2019)G{\"u}hring, Kutyniok, and
  Petersen]{guhring2019error}
Ingo G{\"u}hring, Gitta Kutyniok, and Philipp Petersen.
\newblock Error bounds for approximations with deep relu neural networks in $
  w^{s, p} $ norms.
\newblock \emph{arXiv preprint arXiv:1902.07896}, 2019.
\newblock To appear in "Anal. and Appl.".

\bibitem[Jenatton et~al.(2012)Jenatton, Gribonval, and Bach]{Jenatton}
Rodolphe Jenatton, R{\'e}mi Gribonval, and Francis Bach.
\newblock Local stability and robustness of sparse dictionary learning in the
  presence of noise.
\newblock arxiv, 2012.

\bibitem[Kawaguchi(2016)]{kawaguchi2016deep}
Kenji Kawaguchi.
\newblock Deep learning without poor local minima.
\newblock In \emph{Advances in Neural Information Processing Systems}, pages
  586--594, 2016.

\bibitem[Laurberg et~al.(2008)Laurberg, Christensen, Plumbley, Hansen, and
  Jensen]{laurberg2008theorems}
Hans Laurberg, Mads~Gr{\ae}sb{\o}ll Christensen, Mark~D Plumbley, Lars~Kai
  Hansen, and S{\o}ren~Holdt Jensen.
\newblock Theorems on positive data: On the uniqueness of nmf.
\newblock \emph{Computational intelligence and neuroscience}, 2008, 2008.

\bibitem[Lee and Seung(1999)]{lee1999learning}
Daniel~D Lee and H~Sebastian Seung.
\newblock Learning the parts of objects by non-negative matrix factorization.
\newblock \emph{Nature}, 401\penalty0 (6755):\penalty0 788--791, 1999.

\bibitem[Lee et~al.(2008)Lee, Ekanadham, and Ng]{lee2008sparse}
Honglak Lee, Chaitanya Ekanadham, and Andrew~Y Ng.
\newblock Sparse deep belief net model for visual area v2.
\newblock In \emph{Advances in neural information processing systems}, pages
  873--880, 2008.

\bibitem[Li et~al.(2016)Li, Ling, Strohmer, and
  Wei]{DBLP:journals/corr/LiLSW16}
Xiaodong Li, Shuyang Ling, Thomas Strohmer, and Ke~Wei.
\newblock Rapid, robust, and reliable blind deconvolution via nonconvex
  optimization.
\newblock \emph{CoRR}, abs/1606.04933, 2016.
\newblock URL \url{http://arxiv.org/abs/1606.04933}.

\bibitem[Li and Yuan(2017)]{li2017convergence}
Yuanzhi Li and Yang Yuan.
\newblock Convergence analysis of two-layer neural networks with relu
  activation.
\newblock In \emph{Advances in Neural Information Processing Systems}, pages
  597--607, 2017.

\bibitem[Liu et~al.(2015)Liu, Wang, Foroosh, Tappen, and Pensky]{liu2015sparse}
Baoyuan Liu, Min Wang, Hassan Foroosh, Marshall Tappen, and Marianna Pensky.
\newblock Sparse convolutional neural networks.
\newblock In \emph{Proceedings of the IEEE Conference on Computer Vision and
  Pattern Recognition}, pages 806--814, 2015.

\bibitem[Louizos et~al.(2018)Louizos, Welling, and Kingma]{louizos2017learning}
Christos Louizos, Max Welling, and Diederik~P Kingma.
\newblock Learning sparse neural networks through $ l\_0 $ regularization.
\newblock In \emph{Internatinal Conference on Learning Representation}, 2018.

\bibitem[Malgouyres and Landsberg(2016)]{MalgouyresLandsbergITW}
Fran\c{c}ois Malgouyres and Joseph Landsberg.
\newblock On the identifiability and stable recovery of deep/multi-layer
  structured matrix factorization.
\newblock In \emph{IEEE, Info. Theory Workshop}, Sept. 2016.

\bibitem[Malgouyres and Landsberg(2019)]{MalgouyresLandsberg_long}
Fran\c{c}ois Malgouyres and Joseph Landsberg.
\newblock Multilinear compressive sensing and an application to convolutional
  linear networks.
\newblock \emph{SIAM Journal on Mathematics of Data Science}, 1\penalty0
  (3):\penalty0 446--475, 2019.

\bibitem[Mallat(1998)]{Mallatbook}
St\'ephane Mallat.
\newblock \emph{A Wavelet Tour of Signal Processing}.
\newblock Academic Press, Boston, 1998.

\bibitem[Petersen et~al.(2019)Petersen, Voigtlaender, and Raslan]{Petersen2020}
Philipp Petersen, Felix Voigtlaender, and Mones Raslan.
\newblock Topological properties of the set of functions generated by neural
  networks of fixed size.
\newblock \emph{arXiv preprint 1806.08459}, 2019.
\newblock To appear in "Fondation of computational Mathematics".

\bibitem[Ranzato et~al.(2007)Ranzato, Poultney, Chopra, and
  Cun]{ranzato2007efficient}
Marc'Aurelio Ranzato, Christopher Poultney, Sumit Chopra, and Yann~L Cun.
\newblock Efficient learning of sparse representations with an energy-based
  model.
\newblock In \emph{Advances in neural information processing systems}, pages
  1137--1144, 2007.

\bibitem[Ranzato et~al.(2008)Ranzato, Boureau, and Cun]{ranzato2008sparse}
Marc'Aurelio Ranzato, Y-Lan Boureau, and Yann~L Cun.
\newblock Sparse feature learning for deep belief networks.
\newblock In \emph{Advances in neural information processing systems}, pages
  1185--1192, 2008.

\bibitem[Sedghi and Anandkumar(2014)]{sedghi2014provable}
Hanie Sedghi and Anima Anandkumar.
\newblock Provable methods for training neural networks with sparse
  connectivity.
\newblock In \emph{Deep Learning and representation learning workshop: NIPS},
  2014.

\bibitem[Srinivas et~al.(2017)Srinivas, Subramanya, and
  Venkatesh~Babu]{srinivas2017training}
Suraj Srinivas, Akshayvarun Subramanya, and R~Venkatesh~Babu.
\newblock Training sparse neural networks.
\newblock In \emph{Proceedings of the IEEE Conference on Computer Vision and
  Pattern Recognition Workshops}, pages 138--145, 2017.

\bibitem[Zhang et~al.(2016)Zhang, Du, Zhang, Lan, Liu, Li, Guo, Chen, and
  Chen]{zhang2016cambricon}
Shijin Zhang, Zidong Du, Lei Zhang, Huiying Lan, Shaoli Liu, Ling Li, Qi~Guo,
  Tianshi Chen, and Yunji Chen.
\newblock Cambricon-x: An accelerator for sparse neural networks.
\newblock In \emph{The 49th Annual IEEE/ACM International Symposium on
  Microarchitecture}, page~20. IEEE Press, 2016.

\bibitem[Zhong et~al.(2017)Zhong, Song, Jain, Bartlett, and
  Dhillon]{zhong2017recovery}
Kai Zhong, Zhao Song, Prateek Jain, Peter~L. Bartlett, and Inderjit~S. Dhillon.
\newblock Recovery guarantees for one-hidden-layer neural networks.
\newblock In \emph{Proceedings of the 34th International Conference on Machine
  Learning}, volume~70 of \emph{Proceedings of Machine Learning Research},
  pages 4140--4149, 2017.

\end{thebibliography}
